\documentclass[journal,10pt]{IEEEtran}                                                          

\IEEEoverridecommandlockouts                              

\usepackage[draft]{hyperref}
\usepackage{graphicx,epsfig,tikz,pgf,cancel,color,float,ifpdf,amssymb,amsmath, amsbsy,amsopn, amstext, todonotes}
\usepackage{latexsym,amssymb,amsmath, amsbsy,amsopn, amstext,amsthm}
\usepackage{graphicx,epsfig,tikz,pgf,hyperref,cancel,color,float,ifpdf}
\usepackage{extarrows,mathrsfs}
\usepackage{amsmath,  amssymb,stfloats}
\usepackage{subfigure,setspace}
\newcommand{\vertiii}[1]{{\left\vert\kern-0.25ex\left\vert\kern-0.25ex\left\vert #1 
		\right\vert\kern-0.25ex\right\vert\kern-0.25ex\right\vert}}

\graphicspath{{./}}

\def\diag{diag}

\def\A{\mathcal{A}}

\def\d{\delta}

\def\D{\Delta}

\def\M{\mathcal{M}}

\def\0{{\bf 0}}

\newcommand{\R}{{\mathbb R}}

\newcommand{\Z}{{\mathbb Z}}

\newtheorem{definition}{\bfseries Definition}
\newtheorem{proposition}{\bfseries Proposition}
\newtheorem{example}{\bfseries Example}
\newtheorem{theorem}{\bfseries Theorem}

\newtheorem{corollary}{\bfseries Corollary}
\newtheorem{lemma}{\bfseries Lemma}

\newtheorem{remark}{\bfseries Remark}

\allowdisplaybreaks[1]

\title{Approximation of the Constrained Joint Spectral Radius via Algebraic Lifting}
\author{Xiangru Xu,  Beh{\c{c}}et A{\c{c}}{\i}kme{\c{s}}e
	\thanks{X. Xu is with the Department of Mechanical Engineering, University of Wisconsin-Madison, Madison, WI, USA. Email: {\tt\small xiangru.xu@wisc.edu}. B. A{\c{c}}{\i}kme{\c{s}}e is with the Department of Aeronautics \& Astronautics, University of Washington, Seattle, WA, USA. Email: {\tt\small behcet@uw.edu}. 

	}
}

\begin{document}
	\maketitle
	\thispagestyle{empty}
	\pagestyle{empty}
\begin{abstract}   
This paper studies the constrained switching (linear) system which is a discrete-time switched linear system whose switching sequences are constrained by a deterministic finite automaton. The stability of a constrained switching system is characterized by its constrained joint spectral radius that is known to be difficult to compute or approximate. 
Using the semi-tensor product of matrices, the matrix-form expression of a constrained switching system is shown to be equivalent to that of a lifted arbitrary switching system. Then the constrained joint/generalized spectral radius of a constrained switching system is proved to be equal to
the joint/generalized spectral radius of its lifted arbitrary switching system which can be approximated by off-the-shelf algorithms. 
\end{abstract}

\section{Introduction}\label{sec:intro}

Consider a finite set of matrices $\mathcal{A}=\{A_1,\dots,A_m\}$ with $A_i\in\mathbb{R}^{n\times n}$, $i\in[m]$ where $[m]:=\{1,2,\dots,m\}$.
Dynamics of the discrete-time switched linear system associated with $\mathcal{A}$ is described as 
\begin{align}\label{eqswitch}
x_{k+1}=A_{\sigma_k}x_k
\end{align}
where $x_k\in\R^n$ is the state, and $\sigma_k\in[m]$ is the switching sequence (or mode) of the system that can be considered as the input.  As there is no constraint on the switching sequence, the system \eqref{eqswitch} is called an \emph{arbitrary switching system} and denoted by $S(\A)$ \cite{philippe2016stability,liberzon2012switching}. We call \eqref{eqswitch}  \emph{stable} if $\lim_{k\rightarrow\infty} x_k={\bf 0}$ for any $x_0\in\mathbb{R}^n$ and any switching sequence $\sigma_0\sigma_1\dots$ (see Definition 1.3 in \cite{jungers2009joint}).

The \emph{joint spectral radius} (JSR) of $\mathcal{A}$ is defined as 
\begin{align}
\rho(\mathcal{A})=\limsup_{k\rightarrow\infty}\rho_k(\A)^{1/k}\label{JSR}
\end{align}
where 
\begin{align}
\rho_k(\A)=\max_{\sigma\in[m]^k}\|A_\sigma\|,\label{JSRk}
\end{align}
$\sigma=\sigma_0\dots\sigma_{k-1}$ is a switching sequence of length $k$ with $\sigma_0,\dots,\sigma_{k-1}\in[m]$, $A_\sigma:=A_{\sigma_{k-1}}\dots A_{\sigma_0}$ is the product of $k$ matrices, and $\|\cdot\|$ is any given sub-multiplicative matrix norm on $\R^{n\times n}$. The concept of JSR for a finite set of matrices is a natural generalization of the spectral radius for a single (square) matrix, and was first introduced  in \cite{rota1960note}. Because of the equivalence of matrix norms in finite-dimensional vector spaces, the value of $\rho(\A)$ is independent of the choice of the matrix norm. The JSR found applications in many  areas such as the continuity of wavelet functions, the capacity of codes, and trackable graphs  \cite{jungers2009joint}. Particularly,  the value of $\rho(\mathcal{A})$ characterizes the stability of the switched system \eqref{eqswitch} as \eqref{eqswitch} is stable if and only if $\rho(\mathcal{A})<1$ (see Corollary 1.1 in \cite{jungers2009joint}). 
However, the value of $\rho(\mathcal{A})$ is notoriously difficult to compute or approximate (see the NP-hardness and the undecidability results in \cite{blondel2000boundedness,tsitsiklis1997lyapunov}).
In the past decade, various methods for approximating $\rho(\mathcal{A})$ have been proposed, such as using branch and bound \cite{gripenberg1996computing}, convex combination  \cite{blondel2005computationally},  lifted polytope \cite{jungers2014lifted},  sum-of-squares \cite{parrilo2008approximation,legat2016generating}, and path-complete graph Lyapunov function \cite{ahmadi2014joint}.

By replacing the norm in \eqref{JSRk} with the spectral radius, Daubechies and Lagarias introduced the concept of \emph{generalized spectral radius} (GSR) of $\mathcal{A}$ in  \cite{daubechies1992sets}. Specifically, the GSR of $\mathcal{A}$ is defined as 
\begin{align}
\bar\rho(\mathcal{A})=\limsup_{k\rightarrow\infty}\bar\rho_k(\A)^{1/k}\label{GSR}
\end{align}
where  
\begin{align}
\bar\rho_k(\A)=\max_{\sigma\in[m]^k}\rho(A_\sigma).\label{GSRk}
\end{align}
The \emph{Berger-Wang Theorem} proves that the JSR and the GSR of $\A$ are equivalent, i.e.,  $\rho(\mathcal{A})=\bar\rho(\mathcal{A})$ (see Theorem 4 in \cite{berger1992bounded}).

The switching sequence  $\sigma$ of the switched system \eqref{eqswitch} can be subject to certain constraints. For instance, the switching sequence needs to satisfy a Markovian-like property, that is, $\sigma_k$ that is allowable is dependent on $\sigma_{k-1}$ \cite{kozyakin2014berger,dai2012gel}, or it needs to be accepted by an automaton \cite{wang2017stability}.   Following  \cite{philippe2016stability}, in this paper, we consider switching sequences that are constrained by a deterministic finite automaton. 
\begin{definition}
A deterministic finite automaton  (DFA) $\mathcal{M}$ is a 3-tuple $(Q,U,f)$ where $Q=\{q_1,q_2,\dots,q_\ell\}$ is a finite set of states,
	$U=\{1,2,\dots,m\}$ is a finite set of input symbols, $f: Q\times U\rightarrow Q$ is a transition function\footnote{The transition function is a \emph{partial function}  that may not be defined for all state-input pairs; without loss of generality, we assume that for each state $q\in Q$ there is at least one $u\in U$ such that $f(q,u)$ is defined.}.
\end{definition} 
For system \eqref{eqswitch}, a finite switching sequence $\sigma=\sigma_1...\sigma_k$ is said to be \emph{accepted} by $\M$ if $\sigma_1,...,\sigma_k\in U$ and there exists a finite state sequence $q_{j_1}q_{j_2}\dots q_{j_{k+1}}$ such that $q_{j_1},q_{j_2},\dots, q_{j_{k+1}}\in Q$ and $q_{j_{i+1}}=f(q_{j_i},\sigma_i)$ are defined for $i=1,...,k$; an infinite switching sequence accepted by $\M$ is defined similarly by taking $k=\infty$ \cite{philippe2016stability,eilenberg1974automata}. The set of switching sequences accepted by $\mathcal{M}$ is denoted by $L(\mathcal{M})$. 
Formally, the \emph{constrained switching system}, denoted as $S(\mathcal{A},\mathcal{M})$, is the  switched linear   system as shown in \eqref{eqswitch} where  $A_i\in\mathcal{A}$ for $i\in[m]$ and the switching sequence $\sigma\in L(\mathcal{M})$.

The concept of JSR can be naturally generalized to the case where the switching sequences are constrained by a DFA. Specifically, the \emph{constrained joint spectral radius} (CJSR) of  $S(\mathcal{A},\mathcal{M})$ is defined as
\begin{align}
\rho(\mathcal{A},\mathcal{M})=\limsup_{k\rightarrow\infty}\rho_k(\A,\mathcal{M})^{1/k}\label{CJSR}
\end{align}
where 
\begin{align}
\rho_k(\A,\mathcal{M})=\max_{\substack{\sigma\in[m]^k\\\sigma\in L(\mathcal{M})}}\|A_\sigma\|.\label{CJSRk}
\end{align} 
Similarly, the \emph{constrained generalized spectral radius} (CGSR) of $S(\mathcal{A},\mathcal{M})$ is defined as 
\begin{align}
\bar\rho(\mathcal{A},\mathcal{M})=\limsup_{k\rightarrow\infty}\bar\rho_k(\A,\mathcal{M})^{1/k}\label{CGSR}
\end{align}
where 
\begin{align}
\bar\rho_k(\A,\mathcal{M})=\max_{\substack{\sigma\in[m]^k\\\sigma\in L(\mathcal{M})}}\rho(A_\sigma).\label{CGSRk}
\end{align}
The value of $\rho(\A,\M)$ is independent of the choice of matrix norm in \eqref{CJSRk}, and it characterizes the stability of the constrained switching system $S(\mathcal{A},\mathcal{M})$ as \emph{$S(\mathcal{A},\mathcal{M})$ is stable if and only if $\rho(\mathcal{A},\mathcal{M})<1$} (see Theorem 1.1 in \cite{philippe2016stability} and  Corollary 2.8 in \cite{dai2012gel}). Due to the constraint on switching sequences, the computation or approximation of $\rho(\A,\M)$ is more difficult than $\rho(\A)$, with only a few results known in the literature:
in \cite{philippe2016stability}, the problem of approximating $\rho(\A,\M)$ was reduced to finding a good multinorm, 
where an arbitrarily accurate approximation can be obtained by solving a semi-definite program and using the quadratic-type multinorm;
in \cite{legat2016generating}, an algorithm that generates a sequence of matrices with asymptotic growth rate close to the CJSR was proposed, based on the dual solution of a sum-of-squares optimization
program. The switched linear system whose switching sequences are constrained by a Muller automaton was considered in \cite{wang2017stability}, where a lifting method based on the Kronecker product was proposed and used to show how different notions of stability are related; the switched linear system whose switching sequences are constrained by a given square matrix was considered in  \cite{dai2012gel} and \cite{kozyakin2014berger}, where the Markovian joint spectral radius was discussed and the Markovian analog of the Berger-Wang formula  was derived.

In this paper, we propose a novel lifting method to approximate the CJSR and CGSR of $S(\mathcal{A},\mathcal{M})$ (see Figure \ref{fig0} for a summary of the main results).
The contributions of the paper are summarized as follows: 
i) we propose a unified matrix-based formulation for the arbitrary switching system and the constrained switching system by using the semi-tensor product (STP) of matrices, and based on this formulation we prove that the matrix expression of $S(\mathcal{A},\mathcal{M})$ is equivalent to that of an arbitrary switching system $S(\mathcal{A}_\M)$ which can be considered as a lifted system of $S(\mathcal{A},\mathcal{M})$; ii) we prove that  $\rho(\mathcal{A},\mathcal{M})=\bar\rho(\mathcal{A},\mathcal{M})=\rho(\mathcal{A}_\M)=\bar\rho(\mathcal{A}_\M)$ which can be seen as a version of the Berger-Wang formula for the constrained switching system.
The equivalence of the four quantities (i.e., $\rho(\mathcal{A},\mathcal{M}),\bar\rho(\mathcal{A},\mathcal{M}),\rho(\mathcal{A}_\M),\bar\rho(\mathcal{A}_\M)$) implies that the approximation of $\rho(\mathcal{A},\mathcal{M})$ or $\bar\rho(\mathcal{A},\mathcal{M})$ can be converted into the approximation of $\rho(\mathcal{A}_\M)$ or $\bar\rho(\mathcal{A}_\M)$ for which many off-the-shelf algorithms can be leveraged.

The remainder of the paper is organized as follows: Section \ref{sec:stp} introduces  some preliminaries about STP; Section \ref{sec:arbitrary} presents the STP-based matrix formulation for  the arbitrary switching system, the DFA and the constrained switching system; Section \ref{sec:constrain} gives the main result which proves the equivalence of the CJSR/CGSR of $S(\mathcal{A},\mathcal{M})$ and the  JSR/GSR of $S(\mathcal{A}_\M)$; and Section \ref{sec:conclu} presents some concluding remarks.

\begin{figure}[!t]
	\centering
	\includegraphics[width=0.5\linewidth]{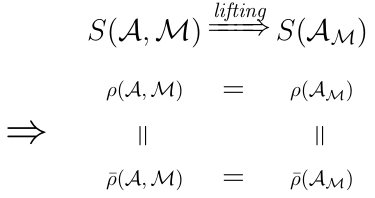}
	\caption{Illustration of the main results.}\label{fig0}
\end{figure}

\section{Preliminaries}\label{sec:stp}
Denote $\R_{>0}$, $\R_{\geq 0}$, $\Z_{>0}$, $\Z_{\geq 0}$ as the sets of positive real numbers, nonnegative real numbers, positive integers, nonnegative integers, respectively. Denote $|\cdot|$ as the cardinality of a set. 
Denote $I_n$ as the $n\times n$ identity matrix, 
$col_i(M)$ as the $i$-th column of matrix $M$, and $Col(M)$ as the set of columns of matrix $M$. Define $\d_n^k:=col_k(I_n)$ where $k\in[n]$, $\d_n^0:=\mathbf{0}_n$ where $\mathbf{0}_n$ is the zero vector of dimension $n$, $\D_n:=\{\d_n^1,\cdots,\d_n^n\}$, $\D_n^e:=\D_n\cup\d_n^0$, and $\d_n[i_1,i_2,\cdots,i_m]:=[\d_n^{i_1},\d_n^{i_2},\cdots,\d_n^{i_m}]\in \mathbb{R}^{n\times m}$ where $\delta_n^{i_j}\in \D_n^e$, $j\in[m]$.

Given two matrices $A\in \mathbb{R}^{m\times n}$ and $B\in \mathbb{R}^{p\times q}$, their conventional matrix product $AB$ requires  $n=p$. The Kronecker product of $A$ and $B$, however, has no such dimensional restriction on $n$ and $p$.

\begin{definition}\label{dkron} 
	Given two matrices $A=(a_{ij})\in \mathbb{R}^{m\times n}$ and $B\in \mathbb{R}^{p\times q}$, their Kronecker product is defined as
	\begin{align*}
	A\otimes B:=\begin{pmatrix}
	a_{11}B & \cdots & a_{1n}B\\
	\vdots & \ddots & \vdots\\
	a_{m1}B & \cdots & a_{mn}B
	\end{pmatrix}.
	\end{align*}
\end{definition}

The following two properties of Kronecker product will be used in later sections \cite{horn1990matrix}:
\begin{itemize}
	\item Given matrices $A\in \mathbb{R}^{m_A\times n_A}$, $B\in \mathbb{R}^{m_B\times n_B}$, $C\in \mathbb{R}^{n_A\times n_C}$, $D\in \mathbb{R}^{n_B\times n_D}$, it holds that 
	\begin{align}\label{prokron}
	(A\otimes B)(C\otimes D)=(AC)\otimes (BD).	
	\end{align} 
	\item Given two matrices $A\in \mathbb{R}^{p\times p}$ and $B\in \mathbb{R}^{q\times q}$, if $\lambda_1,\dots,\lambda_{p}$ are the eigenvalues of $A$ and $\mu_1,\dots,\mu_{q}$ are the eigenvalues of $B$, then the eigenvalues of $A\otimes B$ are $\lambda_i\mu_j$ for $i=1,\dots,p$ and $j=1,\dots,q$.
\end{itemize}

Similar to the Kronecker product, the \emph{semi-tensor product} of matrices can be  defined for two matrices with arbitrary dimensions \cite{cheng2012introduction}.

\begin{definition}\label{dstp} (Def. 1 in \cite{laschov2012controllability}) Given two matrices $A\in \mathbb{R}^{m\times n}$ and $B\in \mathbb{R}^{p\times q}$, their semi-tensor product (STP) is defined as
\begin{align}\label{1.1}
A\ltimes B:=\left(A\otimes I_{s/n}\right)\left(B\otimes I_{s/p}\right)
\end{align}
where $s$ is the least common multiple of $n$ and $p$, and $\otimes$ is the Kronecker product.
\end{definition}

STP degenerates to the conventional matrix product when $n=p$, and it becomes the Kronecker product when $n$ and $p$ are co-prime. 
STP not only has the properties of associativity and distributivity as the conventional matrix product, but also has some unique properties as shown below (see also Page 2 of \cite{cheng2010linear}). 
\begin{itemize}
	\item Given a column vector  $x\in \R^n$ and a matrix $A$, it holds that
	\begin{align}\label{eqprop1}
	x\ltimes A=(I_n\otimes A)\ltimes x.
	\end{align}
	
	\item Given a column vector $x\in\D_n$, there exists a matrix
	$\Phi_{n}=diag(\d_n^1,\d_n^2,\dots,\d_n^n)\in {\mathcal{L}}_{n^2\times n}$ such that
	\begin{align}\label{eqprop2}
	x\ltimes x=\Phi_{n}x.
	\end{align}
	
	\item Given two column vectors
	$x\in \R^{n}$ and $y\in \R^{m}$, there is a matrix
	$W_{[n,m]}= [\d_{m}^1\ltimes\d_{n}^1,\cdots,\d_{m}^{m}\ltimes\d_{n}^1,\cdots,
	\d_{m}^1\ltimes\d_{n}^{n},\cdots,\d_{m}^{m}\ltimes\d_{n}^{n}]\in \mathcal{L}_{mn\times mn}$ such that 
	\begin{align}\label{eqprop3}
	W_{[n,m]}\ltimes x\ltimes y=y\ltimes x.
	\end{align}
	
\end{itemize}
Note that $x$ or $y$ in \eqref{eqprop1}, \eqref{eqprop2} and \eqref{eqprop3} is a given column vector, instead of an indeterminate.

\section{STP Formulation of the Arbitrary and the Constrained Switching System}\label{sec:arbitrary}
\subsection{STP Formulation of the Arbitrary Switching System}\label{subsec:STPASS}

In this subsection, we present a STP-based matrix formulation for the arbitrary switching system $S(\A)$.

Given a finite set of matrices $\mathcal{A}=\{A_1,\dots,A_m\}$ where $A_i\in\mathbb{R}^{n\times n}$, $i\in [m]$, we define a matrix $H$ as 
\begin{align}
H=[A_1,\dots,A_m]\in\mathbb{R}^{n\times nm}.\label{eqH}
\end{align}
Recall that $\delta_m^i$ is a column vector of length $m$ with the only nonzero entry, ``1'', in the $i$-th position. For any $i\in[m]$, we identify $i$ with $\delta_m^i$, denoted as 
\begin{align}
i\sim\delta_m^i.
\end{align}
For any $\sigma_k$, which is the switching sequence of system \eqref{eqswitch} at time step $k$, we define the \emph{vector form} of $\sigma_k$ as a column vector $\sigma(k)\in \Delta_m$ and let $\sigma(k)=\delta_m^i$ when $\sigma_k=i$ where $i\in[m]$. In other words, we identify  $\sigma_k=i\in[m]$ with its vector form $\sigma(k)=\delta_m^i\in\Delta_m$, denoted as 
\begin{align}
\sigma_k\sim \sigma(k).
\end{align}
Define $x(k)\in\mathbb{R}^n$ as the state of system \eqref{eqswitch} by letting $x(k)=x_k$ for any $k\in\Z_{\geq 0}$. In the rest of the paper, we will use $\sigma_k$ and $\sigma(k)$, $x_k$ and $x(k)$ interchangeably when there is no confusion. The \emph{one-to-one correspondence} of the scalar $\sigma_k$ (resp. $x_k$) and the vector $\sigma(k)$ (resp. $x(k)$) is the \emph{key} to converting the algebraic equation as shown in \eqref{eqswitch} into the STP-based matrix formulation as shown in \eqref{dynstpxk}.

\begin{proposition}\label{propdynswt}
	Dynamics of $S(\A)$ as shown in \eqref{eqswitch} can be written equivalent as 
	\begin{align}\label{dynstpxk}
	x(k+1)=H\ltimes \sigma(k)\ltimes x(k)
	\end{align}
	where $H$ is given in \eqref{eqH}, $x(k)\in\R^n$ and $\sigma(k)\in\Delta_m$ are the vector forms of the state and input, respectively.
\end{proposition}
\begin{proof} By the definition of STP, for any $k\in\Z_{\geq 0}$, it holds that $H\ltimes\sigma(k)=A_{\sigma_k}$  where  $\sigma(k)\sim\sigma_k$, $\sigma_k\in[m]$ and $A_{\sigma_k}\in\mathcal{A}$. Hence, \eqref{dynstpxk} is equivalent to $x(k+1)=A_{\sigma_k}\ltimes x(k)=A_{\sigma_k} x(k)$ where STP degenerates to the conventional matrix product as $A_{\sigma_k}\in\mathbb{R}^{n\times n}$ and $x(k)\in \mathbb{R}^{n}$. Note that $x(k+1)=A_{\sigma_k} x(k)$ is just the dynamics of $S(\A)$ as shown in \eqref{eqswitch}, which completes the proof.
\end{proof}

A finite switching sequence $\sigma=\sigma_0\dots\sigma_{k-1}\in[m]^k$ can be expressed equivalently into its \emph{vector form} 
\begin{align}\label{vecsigma1}
\tilde\sigma=\ltimes_{i=0}^{k-1} \sigma(k-1-i)\in\Delta_{m^k}
\end{align} 
where $\sigma_i\sim\sigma(i)$, $i\in\{0,1,\dots,k-1\}$. The following lemma explains the one-to-one correspondence between $\sigma$ and $\tilde\sigma$, denoted as $\sigma\sim\tilde\sigma$. The proof of the lemma can be easily obtained by the definition of STP and is thus omitted.

\begin{lemma}
If $\sigma_i=j_i\sim \delta_m^{j_i}$ where  $j_i\in[m]$, $i\in\{0,1,\dots,k-1\}$, then the sequence $\sigma=\sigma_0\dots\sigma_{k-1}$ is identified with its vector form 
\begin{align}\label{vecsigma2}
\tilde\sigma:=\delta _{m^k}^\tau=\delta_{m}^{j_{k-1}}\ltimes\dots\ltimes\delta_{m}^{j_0}\in\Delta_{m^k}
\end{align}
where
\begin{align}
\tau=1+\Sigma_{i=1}^k(j_{k-i}-1)m^{k-i}\in[m^k].\label{eqtau}
\end{align}
Conversely, given a vector $\tilde\sigma:=\delta _{m^k}^\tau\in\Delta_{m^k}$ where $\tau\in[m^k]$, a set of numbers $j_0,\dots,j_{k-1}\in[m]$ satisfying \eqref{eqtau} can be uniquely  determined, which corresponds to a switching sequence  $\sigma=\sigma_0\dots\sigma_{k-1}\in[m]^k$.	
\end{lemma}

For any $k\geq 2$, from \eqref{dynstpxk} and property \eqref{eqprop1}, it holds that
\begin{align}
x(k)&=H\ltimes\sigma(k-1)\ltimes H\ltimes\sigma(k-2)\nonumber\\
&\quad \quad \quad \ltimes\cdots \ltimes H\ltimes\sigma(0)\ltimes x(0)\nonumber\\
&=\tilde H_k\ltimes_{i=0}^{k-1} \sigma(k-1-i)\ltimes x(0)\label{eqxk}
\end{align}
where 
\begin{align}\label{tilH}
\tilde H_k=H\ltimes_{i=1}^{k-1}(I_{m^i}\otimes H).
\end{align}
Noting that the matrix $\tilde H_k$ in \eqref{tilH} is $n\times nm^k$, we partition it into $m^k$ sub-matrices as follows:
\begin{align}
\tilde H_k=[\tilde H_{k1}\;\tilde H_{k2}\;\dots\;\tilde H_{km^k}]\label{eqHksplit}
\end{align}
where $\tilde H_{ki}\in\mathbb{R}^{n\times n}$, $i\in[m^k]$.
Given an arbitrary finite switching sequence $\sigma=\sigma_0\dots\sigma_{k-1}\in[m]^k$, from \eqref{eqxk} it holds that
\begin{align}\label{eqx}
x(k)=A_\sigma x(0)
\end{align}
where $A_\sigma=A_{\sigma_{k-1}}\dots A_{\sigma_0}\in\R^{n\times n}$. 
If $\sigma_i=j_i$ where  $j_i\in[m]$, $i\in\{0,1,\dots,k-1\}$, then $\tilde H_{k\tau}=A_\sigma$ where $\tau$ is given by \eqref{eqtau}. Hence, $\tilde H_k$ consists of matrices $A_\sigma$ for all possible switching sequences $\sigma$. 

\begin{remark}\label{remark2} 
It is known that $\bar\rho_k(\A)^{1/k}\leq \bar\rho(\mathcal{A})=\rho(\mathcal{A})\leq \rho_k(\A)^{1/k}$ for any $k\in\Z_{>0}$ \cite{jungers2009joint}. 
Since $\rho_k(\A)=\max_{i\in[m^k]}\|\tilde H_{ki}\|$ and $\bar\rho_k(\A)=\max_{i\in[m^k]}\rho(\tilde H_{ki})$ where $\|\cdot\|$ is a sub-multiplicative matrix norm on $\R^{n\times n}$,  
it holds that $\max_{i\in[m^k]}\rho(\tilde H_{ki})^{1/k}\leq \rho(\mathcal{A})\leq \max_{i\in[m^k]}\|\tilde H_{ki}\|^{1/k}$ 	
where $\tilde H_{ki}$ is given in \eqref{eqHksplit} with $\tilde H_{k}$ given in \eqref{tilH}. This inequality, however,  is computationally impractical when $k$ is large since $\tilde H_{k}$ requires computing  all the matrix products of length $k$.
\end{remark}

\begin{example}\label{ex1}
Consider a finite set of matrices $\A=\{A_1,A_2,A_3,A_4\}$ given in  \cite{philippe2016stability}:
\begin{align*}
&A_1=\begin{pmatrix}
0.94 & 0.56 \\-0.35 & 0.73
\end{pmatrix},\;A_2=\begin{pmatrix}
0.94 & 0.56 \\0.14 & 0.73
\end{pmatrix},\\
&A_3=\begin{pmatrix}
0.94 & 0.56 \\-0.35 & 0.46
\end{pmatrix},\;A_4=\begin{pmatrix}
0.94 & 0.56 \\0.14 & 0.46
\end{pmatrix}.
\end{align*}
The STP-based matrix expression \eqref{dynstpxk} for $S(\A)$ is $x(k+1)=H\ltimes \sigma(k)\ltimes x(k)$, where $H=[A_1\;A_2\;A_3\;A_4]$, $x(k)\in\R^2$ and $\sigma(k)\in\Delta_4$. For a switching sequence $\sigma=\sigma_0\sigma_1$ where $\sigma_0=3,\sigma_1=2$, we have $\sigma_0\sim \delta_4^3,\sigma_1\sim \delta_4^2$. Therefore, by \eqref{vecsigma1}, the vector form of $\sigma$ is calculated as $\tilde\sigma=\delta_4^2\ltimes \delta_4^3=\delta_{16}^{7}$ with $\tau=7$ by \eqref{vecsigma2}. Then, $x(2)=A_\sigma x(0)$ by \eqref{eqx} where $A_\sigma=A_3A_2=\tilde H_{27}$,  $\tilde H_2=H\ltimes (I_4\otimes H)$ by \eqref{tilH} and $\tilde H_{27}$ is the 7th sub-matrix of $\tilde H_2$ as defined in \eqref{eqHksplit}.

\end{example}

\subsection{STP Formulation of the DFA}\label{subsec:DFA}
In this subsection we revisit the STP-based matrix expression for the DFA $\mathcal{M}$ (see \cite{xu2012matrix,xu2013matrix,xiangru2013observability} for more detail). 

Consider a DFA $\mathcal{M}=(Q,U,f)$ where $Q=\{q_1,\dots,q_\ell\},\quad U=\{1,\dots,m\}.$ 
Define the \emph{transition structure matrix} of $\mathcal{M}$ as
\begin{align}
F=[F_1\;F_2\;\dots\;F_m]\in \mathbb{R}^{\ell\times m\ell}\label{TSMF}
\end{align}
where $F_j\in \mathbb{R}^{\ell\times \ell}$ is defined as follows: for $j\in[m]$,  
\begin{equation}\label{structmatrix}
{F_j}_{(s,t)}=\begin{cases}
1,\quad\mbox{if} \;q_s= f(q_t,j);\\
0,\quad\mbox{otherwise.}
\end{cases}
\end{equation}
The DFA $\mathcal{M}$ can be seen as a discrete-time dynamical system as follows: 
given an initial state $q_{j_0}$ and an input sequence $\sigma=\sigma_{0}\sigma_{1}\dots$,  $\mathcal{M}$ evolves according to $q_{j_{i+1}}=f(q_{j_i},\sigma_{i})$ if the transition function $f(q_{j_i},\sigma_{i})$ is defined, where $j_0,j_1,\dots\in[\ell],\sigma_0,\sigma_1,\dots\in[m]$. 

Identify each state $q_i\in Q$ with its vector form $\d_\ell^i$ where $i\in [\ell]$ (denoted as $q_i\sim \d_\ell^i$) so that $Q$ is identified with $\D_\ell$. Similarly, identify the input $j\in U$ with its vector form $\d_m^j$ where $j\in [m]$ (denoted as $j\sim \d_m^j$) so that $U$ is identified with $\D_m$. 
Let $q(k)\in\Delta_\ell^e$ and $\sigma(k)\in\Delta_m$ be the \emph{vector forms} of the state and the input of $\mathcal{M}$ at time step $k$, respectively. We let $\sigma(k)=\delta_m^\kappa$  for some $\kappa\in[m]$ if the input $\sigma_k=\delta_m^\kappa$; similarly,  we let $q(k)=\delta_\ell^s$ for some $s\in[\ell]$ if the state $q_k=\delta_\ell^s$ and let $q(k)=\delta_\ell^0$ if the state $q_k$ is undefined. Note that if $f(q_{j_i},\sigma_{i})$ is undefined for some $i\in\Z_{>0}$, then $q_{j_{i+1}}$, $q_{j_{i+2}}$,... are all undefined.

\begin{proposition}(Theorem 1 in \cite{xu2012matrix}) 
	The matrix expression of the dynamics of $\mathcal{M}$ is
	\begin{align}\label{dynstpqk}
	q(k+1)=F\ltimes \sigma(k)\ltimes q(k)
	\end{align}
	where $F$ is defined in \eqref{TSMF}, $q(k)\in\Delta_\ell^e$ and $\sigma(k)\in\Delta_m$ are the vector forms of the state and input of $\mathcal{M}$, respectively.
\end{proposition}
Similar to \eqref{eqxk}, for any $k\geq 2$ we have 
\begin{align}
q(k)&=\tilde F_k\ltimes_{i=0}^{k-1} \sigma(k-1-i)q(0)\label{eqqk}
\end{align}
where $\tilde F_k=F\ltimes_{i=1}^{k-1}(I_{m^i}\otimes F)$. Partition the matrix $\tilde F_k$ into $m^k$ sub-matrices as
$\tilde F_k=[\tilde F_{k1}\;\tilde F_{k2}\;\dots\;\tilde F_{km^k}]$
where $\tilde F_{ki}\in\mathbb{R}^{\ell\times \ell}$, $i\in[m^k]$.
Given an arbitrary switching sequence $\sigma=\sigma_0\dots\sigma_{k-1}\in[m]^k$, from \eqref{dynstpqk} we have 
\begin{align}
q(k)=F_\sigma q(0)
\end{align}
where the matrix $F_\sigma=F_{\sigma_{k-1}}\dots F_{\sigma_0}\in\R^{\ell\times \ell}$ contains the transition information from $q(0)$ to $q(k)$ (refer to \cite{xu2012matrix} for more detail). Recalling the vector form of a switching sequence defined in \eqref{vecsigma1}-\eqref{eqtau}, it is clear that $F_\sigma=\tilde F_{ks}$ when the vector form of $\sigma$ is $\delta_{m^k}^s$ where $s\in[m^k]$. Then the following corollary follows directly.

\begin{corollary}\label{cor1}
	Given a switching sequence $\sigma=\sigma_0\dots\sigma_{k-1}\in[m]^k$, $\sigma\in L(\mathcal{M})$ if and only if  $F_{\sigma_{k-1}}\dots F_{\sigma_0}\neq{\bf 0}$.  
\end{corollary}

\begin{example}\label{ex2}
		\begin{figure}[!h]
		\centering
		\includegraphics[width=0.7\linewidth]{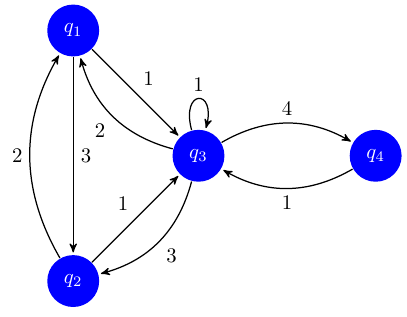}
		\caption{The DFA $\M$ in Example \ref{ex2}.}\label{fig1}
	\end{figure}
	Consider the DFA $\mathcal{M}=(Q,U,f)$ given in Section 4 of \cite{philippe2016stability}, where  $Q=\{q_1,q_2,q_3,q_4\}$, $U=\{1,2,3,4\}$, and its transition map is shown in Fig. \ref{fig1}. 
	The matrix form of the dynamics of $\M$ is given by \eqref{dynstpqk} where $q(k)\in\Delta_4^e$, $\sigma(k)\in\Delta_4$, and $F=[F_1\; F_2\; F_3\; F_4]$ with  $F_1=\delta_4[3,3,3,3]$, $F_2=\delta_4[0,1,1,0]$, $F_3=\delta_4[2,0,2,0]$, $F_4=\delta_4[0,0,4,0]$ (refer to the notations at the beginning of Section \ref{sec:stp}). Given an input sequence $\sigma=231$, its vector form  is $\tilde \sigma:=\delta_4^1\ltimes\delta_4^3\ltimes\delta_4^2=\delta_{64}^{10}$ according to \eqref{vecsigma2} with $\tau=10$ by \eqref{eqtau}. Calculate the matrix $\tilde F_3$, from which the $10$-th block submatrix is $\delta_4[0,3,3,0]$. This submatrix can be interpreted as follows: with the input sequence $\sigma=231$, $\M$ transitions to $q_3$ if it starts from $q_2$ or $q_3$, and the transition is not defined if it starts from $q_1$ or $q_4$. This fact can be easily verified by Fig. \ref{fig1}.

\end{example}

\subsection{STP Formulation of the Constrained Switching System}\label{subsec:STPCSS}
In this subsection, we present a STP-based matrix formulation for the 
constrained switching system $S(\mathcal{A},\mathcal{M})$.

Consider a finite set of matrices $\mathcal{A}=\{A_1,\dots,A_m\}$ where $A_i\in\mathbb{R}^{n\times n}$, $i\in [m]$, and a DFA $\mathcal{M}=(Q,U,f)$ where $|Q|=\ell,|U|=m$. Dynamics of $S(\A)$ and $\mathcal{M}$ are expressed into their respective matrix forms \eqref{dynstpxk}  and \eqref{dynstpqk}, which are restated below.
\begin{align*}
&\mbox{dynamics of}\;S(\A):\quad\; x(k+1)=H\ltimes \sigma(k)\ltimes x(k)\\
&\mbox{dynamics of}\;\mathcal{M}:\quad\quad\;\, q(k+1)=F\ltimes \sigma(k)\ltimes q(k)
\end{align*} 
Define $\xi(k)$ as the state of  $S(\mathcal{A},\mathcal{M})$ at time step $k$ by
\begin{align}
\xi(k)=q(k)\ltimes x(k)\in\R^{n\ell}.\label{dfnxik}
\end{align}
Recalling \eqref{1.1}, it is clear that \eqref{dfnxik} is equivalent to 
\begin{align}\label{xieq}
\xi(k)=q(k)\otimes x(k).
\end{align}
As $\xi$ is a Kronecker product of $q$ and $x$, the state of  $S(\mathcal{A},\mathcal{M})$, $\xi$ can be seen as a lifting of $x$, the state of  $S(\mathcal{A})$. By the definition of \eqref{xieq} and recall that $q(k)\in\Delta_\ell^e$, $x(k)\in\R^{n}$, the block structure of $\xi$ can be  interpreted as follows: if $q(k)=\delta_\ell^0$, then $\xi(k)$ is equal to $\delta_{n\ell}^0$; if $q(k)=\delta_\ell^s$ for some $s\in[\ell]$, then the $s$-th block of $\xi(k)$ is equal to $x(k)$ with all the other blocks  equal to zeros, where each block is a column vector of dimension $n$.

The STP-based matrix expression for $S(\mathcal{A},\mathcal{M})$ is given by the following theorem.

\begin{theorem}\label{thmdynxi}The matrix expression of the dynamics of  $S(\mathcal{A},\mathcal{M})$ is 
	\begin{align}
	\xi(k+1)=\Phi \ltimes\sigma(k)\ltimes\xi(k)\label{eqnxik}
	\end{align}
	where  $\sigma(k)\in\Delta_m$ is the vector form of the input, $\xi(k)\in\R^{n\ell}$ is the  vector form of the state defined in \eqref{dfnxik}, and $\Phi$ is the transition structure matrix defined as
	\begin{align}\label{eqphi0}
	\Phi=[\Phi_1\;\dots\;\Phi_m]\in \mathbb{R}^{n\ell\times mn\ell}
	\end{align}
	with
	\begin{align}
	\Phi_i=F_i\otimes A_i\in \mathbb{R}^{n\ell\times n\ell},\;\forall i\in[m],\label{Phii}
	\end{align}
	and $F_i$ given in \eqref{TSMF}-\eqref{structmatrix}.
\end{theorem}
\begin{proof} Without loss of generality, suppose that the switching sequence at time step $k$ is $\sigma(k)=\delta_m^\kappa$ where $\kappa$ is an arbitrary number satisfying $\kappa\in[m]$. Then, from \eqref{dynstpxk}  and \eqref{dynstpqk},  $q(k+1)=F_\kappa q(k)$ and $x(k+1)=A_\kappa x(k)$ hold.
	Therefore, 
	\begin{align}
	\xi(k+1)&=F_\kappa q(k)\ltimes A_\kappa x(k)\nonumber\\
	&=F_\kappa\ltimes(I_m\otimes A_\kappa)\ltimes q(k)\ltimes x(k)\nonumber\\
	&=(F_\kappa\otimes I_n)(I_m\otimes A_\kappa) \xi(k)\nonumber\\
	&=(F_\kappa\otimes A_\kappa) \xi(k)\nonumber\\
	&=\Phi_\kappa \xi(k)\label{STPxi}
	\end{align}
	where the second equality uses property \eqref{eqprop1}, the third equality uses definition \eqref{1.1}, and the fourth equality uses property \eqref{prokron}. Noting that $\Phi \ltimes\sigma(k)=\Phi_\kappa$, the conclusion follows immediately.  
\end{proof} 
Define a finite set of matrices $\mathcal{A}_\M$ as  
\begin{align}
\mathcal{A}_\M=\{\Phi_1,\dots,\Phi_m\}\label{AM}
\end{align}
where $\Phi_i$ is defined in \eqref{Phii}. 
By Proposition \ref{propdynswt}, dynamics of $S(\mathcal{A}_\M)$, the arbitrary switching system associated with $\mathcal{A}_\M$, is given by 
\begin{align}\label{eqlift}
\tilde x(k+1)=\Phi\ltimes \sigma(k)\ltimes \tilde x(k)
\end{align}
where $\Phi$ is the same as \eqref{eqphi0}, $\tilde x(k)\in\R^{n\ell}$ is the vector form of state of $S(\mathcal{A}_\M)$, and $\sigma(k)\in\Delta_m$ is the vector form of input of $S(\mathcal{A}_\M)$.

Because the same transition structure matrix $\Phi$ is used in both \eqref{eqnxik} and \eqref{eqlift}, the matrix-form dynamical equation of $S(\mathcal{A},\mathcal{M})$ is equivalent to that of $S(\mathcal{A}_\mathcal{M})$. 
Therefore, the arbitrary switching system $S(\mathcal{A}_\mathcal{M})$ can  be considered as a \emph{lifted system} of the constrained switching system $S(\mathcal{A},\mathcal{M})$. This lifting has a clear dynamic system interpretation under the STP framework because it arises from the STP-based formulations of $S(\mathcal{A})$, $\mathcal{M}$ and $S(\mathcal{A},\mathcal{M})$.

\section{Equivalence of the CJSR/CGSR of $S(\mathcal{A},\mathcal{M})$ and the JSR/GSR of $S(\mathcal{A}_\mathcal{M})$}\label{sec:constrain}

Consider a DFA $\mathcal{M}=(Q,U,f)$ where $|Q|=\ell,|U|=m$, a finite set of matrices $\mathcal{A}=\{A_1,\dots,A_m\}$ where $A_i\in\mathbb{R}^{n\times n}$, $i\in [m]$, and a finite set of matrices $\mathcal{A}_\M=\{\Phi_1,\dots,\Phi_m\}$ where $\Phi_i$ is defined in \eqref{Phii}.
In this section, we will prove that the CJSR and CGSR of $S(\mathcal{A},\mathcal{M})$ and the  JSR and GSR of  $S(\mathcal{A}_\M)$ are all equivalent. This is shown in the following theorem,  which is the main result of this paper. 
 
\begin{theorem}\label{thm2}
	The following equality holds: 
	$$
	\rho(\mathcal{A},\mathcal{M})=\bar\rho(\mathcal{A},\mathcal{M})=\rho(\mathcal{A}_\M)=\bar\rho(\mathcal{A}_\M)
	$$
\end{theorem}
\begin{proof} 
	i) For a matrix $S\in\R^{n\ell\times n\ell}$, consider $S$ as a $\ell\times \ell$ block matrix $S=(s_{ij})$ with each block $s_{ij}\in\R^{n\times n}$.  
	Define the following function $\vertiii{\cdot}: \R^{n\ell\times n\ell}\rightarrow \R_{\geq 0}$ (inspired by the function given in \cite{kozyakin2014berger}): 
	\begin{align}
	\vertiii{S}=\max_{j\in[\ell]}\sum_{i=1}^\ell\|s_{ij}\|\label{eqnorm}
	\end{align}
	where  $\|\cdot\|$ is any given sub-multiplicative norm defined on $\R^{n\times n}$. Then we prove that the function $\vertiii{\cdot}$ is a sub-multiplicative norm on $\R^{n\ell\times n\ell}$ similar to the proof given in  \cite{kozyakin2014berger}. Indeed, $\vertiii{\cdot}$ is absolutely homogeneous and positive-definite; furthermore, given two matrices $A=(a_{ij})\in\R^{n\ell\times n\ell}$, $B=(b_{ij})\in\R^{n\ell\times n\ell}$, $\vertiii{\cdot}$ is sub-additive can be easily seen from the fact that   $\|a_{ij}+b_{ij}\|\leq \|a_{ij}\|+\|b_{ij}\|$ for any $i,j\in[\ell]$, and $\vertiii{\cdot}$ is sub-multiplicative because
	\begin{align*}
	\vertiii{AB}=&\max_{j\in[\ell]}\sum_{i=1}^\ell\|\sum_{k=1}^\ell a_{ik}b_{kj}\|\\
	\leq&\max_{j\in[\ell]}\sum_{k=1}^\ell\sum_{i=1}^\ell\|a_{ik}\|\|b_{kj}\|\\
	=&\max_{j\in[\ell]}\sum_{k=1}^\ell\left(\left(\sum_{i=1}^\ell\|a_{ik}\|\right)\|b_{kj}\|\right)\\
	\leq&\max_{j\in[\ell]}\sum_{k=1}^\ell\vertiii{A}\|b_{kj}\|\\
	=&\vertiii{A}\left(\max_{j\in[\ell]}\sum_{k=1}^\ell\|b_{kj}\|\right)\\
	=&\vertiii{A}\vertiii{B}.
	\end{align*} 
	ii) Recalling \eqref{JSRk} and \eqref{CJSRk}, $\rho_k(\A_\M)$ is defined as $\rho_k(\A_\M)=\max_{\sigma\in[m]^k}\vertiii{\Phi_\sigma}$ and $\rho_k(\A,\mathcal{M})$ is defined as $\rho_k(\A,\mathcal{M})=\max_{\sigma\in[m]^k,\sigma\in L(\mathcal{M})}\|A_\sigma\|$,
	where $\Phi_\sigma=\Phi_{\sigma_{k-1}}\dots \Phi_{\sigma_0}$, $A_\sigma=A_{\sigma_{k-1}}\dots A_{\sigma_0}$ and $\sigma=\sigma_0\dots\sigma_{k-1}$ is any switching sequence of length $k$.
	From \eqref{Phii}, we have
	\begin{align}
	\Phi_\sigma=&(F_{\sigma_{k-1}}\otimes A_{\sigma_{k-1}})(F_{\sigma_{k-2}}\otimes A_{\sigma_{k-2}})\dots(F_{\sigma_0}\otimes A_{\sigma_0})\nonumber\\
	=&(F_{\sigma_{k-1}}F_{\sigma_{k-2}}\dots F_{\sigma_0})\otimes (A_{\sigma_{k-1}}A_{\sigma_{k-2}}\dots A_{\sigma_0})\nonumber\\
	=&F_\sigma\otimes A_\sigma\label{eqphi}
	\end{align}
	where $F_\sigma:=F_{\sigma_{k-1}}F_{\sigma_{k-2}}\dots F_{\sigma_0}$. By Corollary \ref{cor1}, $F_\sigma$ has the  property that $F_\sigma\neq{\bf 0}$ if and only if $\sigma\in L(\mathcal{M})$; furthermore, there is at most one entry $``1"$ in each column of $F_\sigma$ with other entries being ``$0$". Then, by  the definition of the norm $\vertiii{\cdot}$ in \eqref{eqnorm}, it is easy to see that  $\vertiii{\Phi_\sigma}= \|A_\sigma\|$ when $\sigma\in L(\mathcal{M})$, and  $\vertiii{\Phi_\sigma}=0$ when $\sigma\notin L(\mathcal{M})$. Hence, it holds that $\rho_k(\A_\M)=\rho_k(\A,\mathcal{M})$ for any $k\in\Z_{k>0}$. Since $\rho(\mathcal{A}_\M)=\limsup_{k\rightarrow\infty}\rho_k(\A_\M)^{1/k}$ and $\rho(\mathcal{A},\mathcal{M})=\limsup_{k\rightarrow\infty}\rho_k(\A,\mathcal{M})^{1/k}$, it follows that
	\begin{align}
	\rho(\A_\M)=\rho(\A,\mathcal{M}).\label{pfeq1}
	\end{align} 
	iii) Recalling \eqref{GSRk} and \eqref{CGSRk}, $\bar\rho_k(\A_\M)$ is defined as 
	$\bar\rho_k(\A_\M)=\max_{\sigma\in[m]^k}\rho(\Phi_\sigma)$ and $\bar\rho_k(\A,\mathcal{M})$ is defined as $\bar\rho_k(\A,\mathcal{M})=\max_{\sigma\in[m]^k,\sigma\in L(\mathcal{M})}\rho(A_\sigma)$. 
	For any input sequence $\sigma$, $F_\sigma$ is a square matrix and has at most one entry $``1"$ in each column with other entries being ``$0$". 
	It is not hard to see that there exist a zero matrix $K_1={\bf 0}$, a strictly lower triangular matrix $K_2$, and 
	a permutation matrix $K_3$, such that $F_\sigma$ is similar to a matrix $K:=\diag(K_1,K_2,K_3)$, denoted as $F_\sigma\sim_s K$, where $K_1,K_2,K_3$ can be absent.
	Define $L^{(per)}(\M)=\{\sigma: \sigma\in L(\M), F_\sigma\sim_s\diag(K_1,K_2,K_3)\;\text{where}\;K_3\;\textit{is present}\}$, where (per) stands for ``permutation matrix". We claim that i) the spectral radius of $F_\sigma$ is equal to $1$ if and only if $\sigma\in L^{(per)}(\M)$; ii) for any $k_0\in\Z_{>0}$, there exist some $k'_0>k_0$ and $\sigma\in L^{(per)}(\M)$ such that $|\sigma|=k_0'$. 
	Since the eigenvalues of a permutation matrix lie on the unit circle, the absolute value of the eigenvalues of a permutation matrix are equal to $1$. Hence, given $F_\sigma$ where $\sigma\in L^{(per)}(\M)$, the absolute value of its eigenvalues are either equal to $1$ or $0$, where the latter case happens when $F_\sigma\sim_s\diag(K_1,K_2,K_3)$ with $K_1$ or $K_2$ present;  given $F_\sigma$ where $\sigma\in L(\M),\sigma\notin L^{(per)}(\M)$, all of its eigenvalues are equal to $0$. The first claim is thus proved. Given any $k_0\in\Z_{>0}$, there always exists some $k_0'$ with $k'_0>k_0$ such that $\M(V,E)$ has a loop of length $k_0'$, because $\M$ is assumed to be alive. Hence, there exists some state $q\in Q$ and input sequence $\sigma\in L(\M)$ such that $|\sigma|=k_0'$ and $f(q,\sigma)=q$. This means that there is at least a ``$1$" in the diagonal of $F_\sigma$, which implies that $\sigma\in L^{(per)}(\M)$ by definition. The second claim is thus proved.
	
	Define $\bar\rho_k^{(per)}(\A,\mathcal{M})=\max_{\sigma\in[m]^k,\sigma\in L^{(per)}(\mathcal{M})}\rho(A_\sigma)$, and $\bar\rho^{(per)}(\mathcal{A},\mathcal{M})=\limsup_{k\rightarrow\infty}\bar\rho_k^{(per)}(\A,\mathcal{M})^{1/k}$. Recalling the second property of the Kronecker product in Section \ref{sec:stp}, it is easy to see that $\bar\rho_k(\A_\M)=\bar\rho_k^{(per)}(\A,\mathcal{M})$. Then it follows that $\bar\rho(\A_\M)=\bar\rho^{(per)}(\A,\mathcal{M})$. 
	By the Berger-Wang Theorem,  it holds that 
	\begin{align}
	\rho(\A_\M)=\bar\rho(\A_\M).\label{pfeq2}
	\end{align}
	iv) Therefore, $\bar\rho^{(per)}(\A,\mathcal{M})=\rho(\A,\mathcal{M})$. Since $L^{(per)}(\M)\subseteq L(\M)$, it holds that $\bar\rho^{(per)}(\A,\M)\leq\bar \rho(\A,\M)$. By definition, it holds that $\bar\rho(\A,\M)\leq \rho(\A,\M)$. Therefore, we have 
	\begin{align}
	\bar\rho(\A,\mathcal{M})=\rho(\A,\mathcal{M}).\label{pfeq3}
	\end{align}
The proof is complete by combining \eqref{pfeq1}, \eqref{pfeq2} and \eqref{pfeq3}.
\end{proof}	

Theorem \ref{thm2} can be seen as a version of the Berger-Wang formula for the constrained switching system \cite{berger1992bounded}. The importance of Theorem \ref{thm2} lies in that the problem of approximating the CJSR/CGSR of $S(\mathcal{A},\mathcal{M})$ can be converted into the problem of approximating the JSR/GSR of its lifted system $S(\mathcal{A}_\M)$, for which many off-the-shelf algorithms exist\footnote{Different algorithms to compute the JSR/GSR of an arbitrary switching system was summarized in \cite{vankeerberghen2014jsr} where a Matlab toolbox was also provided. }.

\begin{example}\label{ex3}
	Consider a constrained switching system $S(\A,\M)$ where the set $\A$ is given in Example \ref{ex1} and the DFA $\M$ is given in Example \ref{ex2}.  For $i\in[4]$, calculate matrices $\Phi_i=F_i\otimes A_i$ where $F_i$ and $A_i$ are given in Example \ref{ex1} and Example \ref{ex2}, respectively. Define the set of matrices $\A_\M=\{\Phi_1,\Phi_2,\Phi_3,\Phi_4\}$. We use the JSR toolbox in \cite{vankeerberghen2014jsr} to approximate the value of $\rho(\mathcal{A}_\M)$, or equivalently, $\rho(\mathcal{A},\mathcal{M})$, by the conclusion of Theorem \ref{thm2}. In a computer with 3.5GHz CPU and 16GB memory, it takes about \underline{13.7 seconds} for the \emph{jsr} function to return the following bounds: 
	$$
	0.974817198\leq \rho(\mathcal{A},\mathcal{M})\leq 0.974817295
	$$ 
	where the Gripenberg's algorithm and the conitope algorithm are utilized \cite{jungers2014lifted,gripenberg1996computing}.

	In \cite{philippe2016stability}, the multinorm method was proposed to
	approximate the CJSR of a constrained switching system,  where several algebraic lifting methods (e.g. the T-product lift, the M-path-dependent lift) were also combined to improve the estimation accuracy.  As pointed out in \cite{philippe2016stability}, the number of edges in the $T$-product lift of $S$ increases exponentially with T and the number of nodes and edges in the $M$-path-dependent lift of $S$ both grow exponentially with $M$. Figure 7 in \cite{philippe2016stability} shows that execution times for producing CJSR estimates using the $T$-Product lift (or the $M$-path-dependent lift) grow exponentially with $T$ (or $M$). Therefore, estimation of CJSR using the $T$-product lift  or the $M$-path-dependent lift is restricted to small $T$ or $M$.	
	Using the toolbox provided in \cite{philippe2016stability} and fixing $T=8$, it takes about \underline{581 seconds} (on the same computer as above) for the T-product lift to obtain the following bounds: 
	$$
	0.9289\leq \rho(\mathcal{A},\mathcal{M})\leq 0.9761,
	$$ 
	while by fixing $M=7$, it takes about \underline{1156 seconds} for the M-path-dependent lift to obtain the following bounds: 
	$$
	0.9277\leq \rho(\mathcal{A},\mathcal{M})\leq  0.9748.
	$$
	Our lifting method clearly  returns a more accurate approximation of $\rho(\mathcal{A},\mathcal{M})$ in a much shorter time.  Compared with the method in \cite{philippe2016stability}, our method provides a promising computational approach for approximating CJSR/CGSR. 
\end{example}

\begin{remark}
In \cite{wang2017stability}, a lifting method based on the Kronecker product was proposed to study different notions of stability of a constrained switching system. 
Given a finite set of matrices $\mathcal{A}=\{A_1,\dots,A_m\}$ where $A_i\in\mathbb{R}^{n\times n}$, $i\in [m]$, and a 
DFA\footnote{A DFA $\mathcal{M}$ can be considered as a directed and labeled graph $\mathcal{M}(V,E)$ where $V$ is the set of nodes and $E$ is the set of edges. The edge $(v,w,s)\in E$ if and only if there is an edge from the  node $v$ to $w$ via the label $s$. See page 2 in \cite{wang2017stability} for more detail.} $\M(V,E)$ where $V=\{v_1,\dots,v_{\ell}\}$, the lifting of $\A$ and $\M$ proposed in \cite{wang2017stability} is defined as a finite set of matrices 
\begin{align}\label{tilam}
\tilde \A_{\M}:=\{A_{(v_s,v_t,j)}:(v_s,v_t,j)\in E\}
\end{align}
where $A_{(v_s,v_t,j)}:=(\delta_\ell^t(\delta_\ell^s)^\top)\otimes A_j$ 
and $\otimes$ is the Kronecker product. 
It is easily seen that the matrix $\Phi_i$ defined in \eqref{Phii} has the same sizes as $A_{(v_s,v_t,j)}$ and 
$F_j=\sum_{\{s,t:(v_s,v_t,j)\in E\}}\delta_\ell^t(\delta_\ell^s)^\top$ for any $j\in[m]$, 
where the matrix $F_j$ is defined in \eqref{structmatrix}. Therefore, 
\begin{align}
\Phi_j=\sum_{\{s,t:(v_s,v_t,j)\in E\}}A_{(v_s,v_t,j)},\;\forall j\in[m],
\end{align}
which implies that $|\A_\M|\leq |\tilde \A_\M|$ where $\A_\M$ is defined in \eqref{AM} and $\tilde \A_\M$ is defined in \eqref{tilam}. 
Compared with \cite{wang2017stability}, the lifting method proposed in the present paper is much more compact as the size of the set of lifted matrices can be significantly reduced. 
\end{remark}

\begin{remark}
In \cite{kozyakin2014berger}, a lifting method based on the Kronecker product was proposed to investigate the Markovian-type JSR/GSR of a constrained switching system.  
Given a finite set of matrices $\mathcal{A}=\{A_1,\dots,A_m\}$ where $A_i\in\mathbb{R}^{n\times n}$, $i\in[m]$ and a matrix $\Omega=(\omega_{ij})\in\R^{m\times m}$ where $\omega_{ij}\in\{0,1\}$, the matrix product $A_{\sigma_k}\dots A_{\sigma_1}$ for $k\geq 2$ is called \emph{Markovian} if each pair of indices $\{\sigma_{i},\sigma_{i+1}\}$ is $\Omega$-admissible, i.e., $\omega_{\sigma_{i+1}\sigma_{i}}=1$ for all $i\in[k-1]$. The \emph{Markovian joint spectral radius} of $\A$ and $\Omega$ is defined in \cite{kozyakin2014berger} as 
\begin{align}
\rho(\A,\Omega):=\limsup_{k\rightarrow \infty}\rho_k(\A,\Omega)^{1/k}
\end{align}
where 
$\rho_k(\A,\Omega)=\max\{\|A_{\sigma_k}\dots A_{\sigma_1}\|: \omega_{\sigma_{i+1}\sigma_{i}}=1,\forall i\in[k-1]\}$, 
and the \emph{Markovian generalized spectral radius} of $\A$ and $\Omega$ is defined as
\begin{align}
\bar\rho(\A,\Omega):=\limsup_{k\rightarrow \infty}\bar\rho_k(\A,\Omega)^{1/k}
\end{align} 
where 
$\bar\rho_k(\A,\Omega)=\max\{\rho(A_{\sigma_k}\dots A_{\sigma_1}): \omega_{\sigma_{i+1}\sigma_{i}}=1,\forall i\in[k-1]\}$. 
The $\Omega$-lift of $\A$, which is defined on page 5 of \cite{kozyakin2014berger}, is a finite set of matrices\footnote{By using the lifted system $\A_{\Omega}$, the Markovian analog of the Berger-Wang formula was proved in \cite{kozyakin2014berger}: the Markovian joint spectral radius of $\A$ and $\Omega$ is equivalent to the Markovian generalized spectral radius of $\A$ and $\Omega$, which is also equivalent to the joint/generalized spectral radius of $\A_{\Omega}$.} 
\begin{align}\label{Aomeg}
\A_{\Omega}:=\{\Omega_i\otimes A_i: i\in[m]\}
\end{align} 
where $\Omega_i$ is the $i$-th column of $\Omega$. 
The Markovian joint spectral radius  discussed in \cite{kozyakin2014berger} is different from the CJSR/CGSR discussed in the present paper, because switching sequences in \cite{kozyakin2014berger} are constrained by a square matrix $\Omega$ while switching sequences in the present paper are constrained by a DFA $\M$. Particularly, a matrix product $A_{\sigma_k}\dots A_{\sigma_1}$ where $\sigma=\sigma_1...\sigma_k\in L(\mathcal{M})$ is not necessarily Markovian because $\sigma$ is dependent on the state of the DFA that it starts from. For instance,  in Example \ref{ex2}, the input sequence $\sigma=231$ is allowed when starting from $q_2$ or $q_3$, but not allowed starting from $q_1$ or $q_4$; this implies that $\Omega$-admissibility in \cite{kozyakin2014berger} is not well-defined for the DFA given in Example 2. Therefore, the method of \cite{kozyakin2014berger} does not apply to the estimation of CJSR/CGSR. 
While the lifted matrices $F_i\otimes A_i$ in $\A_{\M}$ and $\Omega_i\otimes A_i$ in $\A_{\Omega}$ have  similar structures, the matrix $F_i$, which is the $i$-th sub-matrix of the \emph{transition structure matrix} of  $\M$, is different from the matrix $\Omega_i$, which is the $i$-th column of the \emph{square matrix} $\Omega$. Another difference between our method and \cite{kozyakin2014berger} is that equations \eqref{eqnxik} and \eqref{eqlift} provide 
\emph{a dynamic system interpretation} for $\A_{\M}$ under the STP framework. 
\end{remark}

\section{Conclusion}\label{sec:conclu}
In this paper, we proposed a matrix-based formulation for the arbitrary switching system and the constrained switching system using semi-tensor product of matrices. We showed that the matrix expression of a constrained switching system $S(\mathcal{A},\mathcal{M})$ is equivalent to that of a lifted arbitrary switching system $S(\mathcal{A}_\mathcal{M})$. We proved that the constrained joint/generalized spectral radius of a constrained switching system $S(\mathcal{A},\mathcal{M})$ is equivalent to the joint/generalized spectral radius of the lifted arbitrary switching system $S(\mathcal{A}_\mathcal{M})$. Therefore, off-the-shelf algorithms for approximating the joint/generalized spectral radius can be utilized directly to approximate the constrained joint/generalized spectral radius. In future work, we plan to develop more efficient algorithms for approximating CJSR by incorporating the proposed lifting method and other lifting methods in the literature.

\bibliographystyle{IEEEtran}
\bibliography{./stpswitching}

\end{document}